\documentclass[12pt,leqno]{amsart}

\usepackage{amsmath,amsfonts,amssymb,amsthm,mathrsfs}
\usepackage{curves}  
\usepackage{epic}  
\usepackage[all]{xy}   
\usepackage{graphicx}
\usepackage{epsfig}

\usepackage[usenames]{color}

\newtheorem{thm}{Theorem}[section]

\newtheorem{cor}[thm]{Corollary}
\newtheorem{lem}[thm]{Lemma}

\newtheorem{example}[thm]{Example}

\newtheorem{rem}[thm]{Remark}

\DeclareMathOperator{\diam}{diam}

      \newcommand{\N}{{\mathbb N}}
      \newcommand{\R}{{\mathbb R}}

\def\dist{\qopname\relax o{dist}}
\def\b{\qopname\relax o{b}}

\title[Characterizations to the fractional \\Sobolev inequality]{Characterizations to the fractional \\Sobolev inequality}

\author{Ritva Hurri-Syrj\"anen and Antti V. V\"ah\"akangas}
\address{Department of Mathematics and Statistics, 
Gustaf H\"allstr\"omin katu 2$\b$, FI-00014 University of Helsinki, Finland}
\email{ritva.hurri-syrjanen@helsinki.fi}
\email{antti.vahakangas@helsinki.fi}
\date{\today}

\begin{document}

\begin{abstract} We characterize the fractional Sobolev inequality with fractional isocapacitary and isoperimetric inequalities. 
We give a sufficient condition and examples so that the fractional capacity of the closure of an open set is bounded above by the fractional perimeter of its interior. 
\end{abstract}


\maketitle

\markboth{\textsc{R. Hurri-Syrj\"anen and A. V. V\"ah\"akangas}}
{\textsc{Characterizations to the fractional Sobolev inequality}}

\section{Introduction}

Let $G$ be an open set in $\R^n$ and $\delta \in (0,1)$ be given. If $1\le p <n/\delta$ and if there is a constant $C$ such that
the inequality
\begin{equation}\label{sobolev}
\bigg(\int_G \lvert u(x)\rvert^{np/(n-\delta p)}\,dx\bigg)^{(n-\delta p)/np} \le C
 \bigg( \int_G\int_G\frac{\lvert u(x)-u(y)\rvert^p}{\lvert x-y\rvert^{n+\delta p}}\,
dy\,dx\,\bigg)^{1/p}
\end{equation}
holds for all measurable functions $u:G\to\R$ with  compact support in $G$, then inequality
\eqref{sobolev} is called a fractional Sobolev inequality. In the case $p=1$ we characterize this inequality with
the fractional $(\delta ,1)$-capacity
$\mathrm{cap}_{\delta,1}(\cdot ,G)$
and with the fractional $\delta$-perimeter
$P_\delta(\cdot ,G)$. 
For the definitions we refer to Section
\ref{notation}.

We state our characterization theorem.

\begin{thm}\label{t.equiv}
Suppose that  $G$ is an open set in $\R^n$. Let
$\delta\in (0,1)$ and a constant $C>0$ be given. Then the following conditions are equivalent.
\begin{itemize}
\item[(A)]  The fractional Sobolev inequality
\[
\bigg(\int_G \lvert u(x)\rvert^{n/(n-\delta)}\,dx\bigg)^{(n-\delta)/n} \le C 
\int_G\int_G\frac{\lvert u(x)-u(y)\rvert}{\lvert x-y\rvert^{n+\delta}}\,
dy\,dx
\]
holds for all measurable functions $u:G\to \R$ with compact support in $G$.
\item[(B)] The fractional isocapacitary inequality
\[
\lvert K\rvert^{(n-\delta)/n} \le C\,\mathrm{cap}_{\delta,1}(K,G)
\]
holds for every compact set $K$ in $G$.
\item[(C)] 
The fractional isoperimetric inequality
\[
\lvert D\rvert^{(n-\delta)/n} \le 2C\,P_\delta(D,G)
\]
holds for every  open set $D\subset\subset G$ whose boundary $\partial D$  is an $(n-1)$-dimensional
$C^\infty$-manifold in $G$. 
\end{itemize}
\end{thm}

Theorem \ref{t.equiv} is a corollary of a more general result in Section
\ref{m_proof}.  We emphasize that the best constant is the same in each inequality of cases (A), (B), and (C). Our motivation has been the work of Vladimir Maz'ya on the equivalence of the classical Sobolev type inequalities and the classical isoperimetric and isocapacitary inequalities, \cite{Maz2}, \cite{Maz}.

Fractional isoperimetric inequalities have been studied,
for example, by Rupert L.  Frank and Robert Seiringer in \cite{FS} and their stability versions by Nicola Fusco, Vincent Millot, and Massimiliano Morini in \cite{FMM}.
For an example of a related non-fractional case we refer to
\cite{CFMP}.
We prove that
the sets $D\subset\subset G$ from case (C) are examples of the sets which satisfy the following inequality
\begin{equation}\label{capacity_perimeter_ineq}
\mathrm{cap}_{\delta,1}(\overline{D},G) \le 2P_\delta(D,G)\,,
\end{equation}
we refer to Section \ref{s.on}. 
Quasiballs are also examples of these sets when $G=\R^n$, Example~\ref{e.ex}. 
We give a sufficient condition for sets
to satisfy inequality \eqref{capacity_perimeter_ineq} in
Theorem \ref{l.regular}. We note that the left hand side of inequality \eqref{capacity_perimeter_ineq} may be viewed as a lower
bound  for \[\lvert \chi_{D}\rvert_{W^{\delta,1}(G)}= 2P_\delta(D,G)\,.\]
Hence, this inequality  is related to the question if the characteristic function
$\chi_{D}$ belongs to the fractional homogeneous Sobolev space $\dot{W}^{\delta,1}(G)$,  \cite{Faraco}.

\section{Notation and preliminaries}\label{notation}

Throughout the paper we assume that $G$ is an open set in the Euclidean $n$-space
$\R^n$, $n\geq 2$. 
The open ball centered at $x\in \R^n$ and with radius $r>0$ is  $B^n(x,r)$.
The Euclidean
distance from $x\in G$ to the boundary of $G$ is written as $\dist(x,\partial G)$.
The diameter of a set $A$ in $\R^n$ is $\mathrm{diam}(A)$.
The Lebesgue $n$-measure of a  measurable set $A$ is denoted by $\vert A\vert.$
We write $\chi_A$ for the characteristic function of a set $A$.

The family $C_0(G)$ consists of all continuous functions $u:G\to \R$ with compact support in $G$. 
If $u:G\to \R$ and $t\in\R$ then we write shortly \[\{u>t\}=\{x\in G\,:\,u(x)>t\}\]
and likewise for the sets $\{u=t\}$ and $\{u<t\}$.
We let $C(\ast,\dotsb,\ast)$  denote a constant which depends on the quantities appearing
in the parentheses only.

Let $G$ be an open set in $\R^n$. Let $0< p<\infty$ and $0<\delta<1$ be given. We write
\[
\lvert u \rvert_{W^{\delta,p}(G)} = \bigg( \int_G\int_G\frac{\lvert u(x)-u(y)\rvert^p}{\lvert x-y\rvert^{n+\delta p}}\,
dy\,dx\,\bigg)^{1/p}
\]
for real-valued measurable  functions $u$ on $G$.
The homogeneous fractional Sobolev space
$\dot{W}^{\delta,p}(G)$ consists
of all measurable functions $u:G\to \R$  
with $\lvert u\rvert_{W^{\delta,p}(G)}<\infty$. 

The following lemma from \cite[Lemma 2.6]{H-SV} tells that the functions
$u\in \dot{W}^{\delta,p}(G)$
are locally $L^p$-integrable in $G$ that is, 
$u\in L^p_{\textup{loc}}(G)$.

\begin{lem}\label{l.integrability}
Suppose that $G$ is an open set in $\R^n$. Let $0< p<\infty$ and $0<\delta<1$ be given.
Let $K$ be a compact set in $G$.  
If $u\in \dot{W}^{\delta,p}(G)$, then $u\in L^p(K)$. 
\end{lem}

For a compact set $K$ in $G$, its fractional
 $(\delta,p)$-capacity is the number
\[
\mathrm{cap}_{\delta,p}(K,G) = \inf_u \lvert u\rvert_{W^{\delta,p}(G)}^p\,,
\]
where the infimum is taken over all functions  $u\in C_0(G)$ such that $u(x)\ge 1$ for each $x\in K$.

The fractional $\delta$-perimeter of a given measurable set with respect to $G$ is defined as
\[
P_\delta(A,G) = \int_A \int_{G\setminus A} \frac{1}{\lvert x-y\rvert^{n+\delta}}\,dy\,dx\,.
\]
We note that
\begin{equation*}
P_\delta(A,G)=\frac{1}{2}\lvert \chi_A\rvert_{W^{\delta,1}(G)}\,.
\end{equation*}
Suppose that $A$ is a  non-empty compact set in $\R^n$ and let $s\in [0,n]$. 
The lower $s$-dimensional Minkowski content
 $\mathcal{M}^{s}_\ast(A)$ is defined by 
 \[
  \mathcal{M}^{s}_\ast(A) = \liminf_{r\to 0+} \frac{\lvert A+B^n(0,r)\rvert}{r^{n-s}}\,.
 \]
The $s$-dimensional Hausdorff measure of $A$ is written as $\mathcal{H}^s(A)$.
We recall from \cite[p. 79]{Mat95} that 
 there is a constant $C=C(s,n)>0$ such that
 $\mathcal{H}^{s}(A)\le C\mathcal{M}^{s}_\ast(A)$.

\begin{rem}\label{r.comparion}
Suppose that $A$ is a non-empty compact set in $\R^n$. If $0<\delta<1$ and $\delta <\rho\le n$ are
given such that $\mathcal{H}^{n-\rho}(A)<\infty$, then $\mathcal{H}^{n-\delta}(A)=0$.
We refer to \cite[Theorem 4.7]{Mat95}.
\end{rem}

%
 
 We say that a closed set $\Gamma\subset G$ is an $(n-1)$-dimensional $C^\infty$-manifold in an open set $G$,
if  for each point $x\in \Gamma$ there
 exist open sets $U$ in $ \R^{n-1}$ and $V$ in  $\R$ and
 a smooth function $g:U\to V$ such that $x=\psi(x)\in \psi(U\times V)$ and
 \[
\Gamma \cap \psi(U\times V) =  \{ \psi(y,g(y))\,:\, y\in U\}
 \]
with a rotation $\psi$ about the point $x$.

 \begin{lem}\label{l.sard_type}
If $\Gamma$ is a compact $(n-1)$-dimensional $C^\infty$-manifold in an
 open set $G$ in $\R^n$, then $\mathcal{H}^{n-1}(\Gamma)<\infty$. 
If $u\in C^\infty_0(G)$, then 
 \begin{equation}\label{e.min}
\mathcal{H}^{n-1}(\{x\in G\,:\,u(x)=t\})<\infty
 \end{equation} for almost
 every $t>0$.
 \end{lem}
 
 \begin{proof}
By compactness of $\Gamma$ there are open sets $U_j$ in $\R^{n-1}$ and $V_j$ in  $\R$
and closed balls $ B_j$ in  $U_j$, $j=1,\ldots,N$,  such that
\[
\Gamma = \bigcup_{j=1}^N \{Ê\psi_j(y,g_j(y))\,:\, y\in  B_j\}
\]
for rotations $\psi_j$ about points $x_j$ in $\Gamma$.
Since the function $y\mapsto \psi_j(y,g_j(y))$ is $L$-Lipschitz on $B_j\subset\subset U_j$ with some constant $L$, by Kirszbraun theorem on Lipschitz extension
there exists an $L$-Lipschitz function $f_j:\R^{n-1}\to \R^n$ with $f_j(y)=\psi_j(y,g_j(y))$ if $y\in B_j$.
Hence, we have
\begin{equation*}
\mathcal{H}^{n-1}(
f_j( B_j)) \le  L^{n-1}\mathcal{H}^{n-1}( B_j)<\infty\,.
\end{equation*}
By subadditivity
we obtain 
\[
\mathcal{H}^{n-1}(\Gamma) 
\le \sum_{j=1}^N \mathcal{H}^{n-1}(f_j( B_j))<\infty\,.
\] 
Inequality \eqref{e.min} for almost every $t>0$ is a consequence of Sard's theorem \cite{Sard}
which, together with the implicit function theorem, implies that the level set $\{ x\in G\,:\, u(x)=t\}$
is a compact $(n-1)$-dimensional $C^\infty$-manifold in $G$ for almost every $t>0$.
\end{proof}

\section{On the fractional capacity and perimeter}\label{s.on}

Suppose that $G$ is an open set in $\R^n$ and that $0<\delta<1$ is given.  
Theorem \ref{l.regular} gives a suffcient condition for open sets $D\subset\subset G$ in order that
the inequality 
\begin{equation}\label{frac_cap_per}
\mathrm{cap}_{\delta,1}(\overline{D},G) \le 2P_\delta(D,G)
\end{equation}
holds.
For concrete examples of the sets which satisfy inequality \eqref{frac_cap_per} we refer 
to Example~\ref{e.ex}. 

\begin{thm}\label{l.regular}
Suppose that $G$ is an open set in $\R^n$ and $0<\delta<1$.
If $D\subset\subset G$ is an open set such that $\mathcal{H}^{n-\delta}(\partial D)=0$, 
then inequality \eqref{frac_cap_per} holds with respect to $G$.
\end{thm}

For the proof of Theorem \ref{l.regular} we need an auxiliary result.

\begin{lem}\label{l.regular_est}
Suppose that $G$ is an open set in $\R^n$ and let $0<\delta<1$ be given.
Let $D\subset\subset G$ be an open set such that $\mathcal{H}^{n-\delta}(\partial D)=0$.
Let $\varepsilon>0$ be given.
Then, there exists a function $u$ in $C_0(G)$
such that $0\le u\le 1$ and $u(x)= 1$ for every  $x\in \overline{D}$.
Moreover,
\begin{equation}\label{e.limit}
 \int_{G\setminus D}\int_{G\setminus D}\frac{\lvert u(x)-u(y)\rvert}{\lvert x-y\rvert^{n+\delta}}\,dy\,dx <\varepsilon\,.
\end{equation}
\end{lem}

\begin{proof}
We may assume that $G\not=\R^n$. If $G=\R^n$, then we just remove one point from $G\setminus \overline{D}$.
Let us fix a non-negative $\psi\in C^\infty_0(B^n(0,2))$ with $\psi(x)=1$ for every $x\in B^n(0,1)$.
For a given ball $B=B^n(x_B,r_B)$ and $x\in\R^n$â we write $\psi^B(x)=\psi( (x-x_B)/r_B)$. We note
that $\psi^B(x)=1$ for every $x\in B$ and $\psi^B(x)=0$ if $x\in \R^n\setminus 2B$.
By a change of variables, we find that
\begin{equation}\label{e.small}
\lvert \psi^B \rvert_{W^{\delta,1}(\R^n)} = r_B^{n-\delta}Ê\lvert \psi\rvert_{W^{\delta,1}(\R^n)}<\infty\,.
\end{equation}
Let us fix $\varepsilon\in (0,\dist(\partial D,\partial G)/6)$. Because $\mathcal{H}^{n-\delta}(\partial D)=0$, there
are sets $E_1,E_2,\ldots$ in $\R^n$ such that 
\[
\partial D\subset \bigcup_{i} E_i\quad \text{ andÊ} \quad \sum_{i} \diam(E_i)^{n-\delta} < \min\{ \varepsilon,
(\dist(\partial D, \partial G)/6)^{n-\delta}\}\,.
\]
We refer to \cite[Lemma 4.6]{Mat95}.

Let $i\in \N$.
Without loss of generality, we may assume that there is a point $x_i\in E_i\cap \partial D$.
We write $B_i = B^n(x_i,2\diam(E_i)+2^{-i}\varepsilon)$. Hence, $E_i\subset B_i$.
Then the family
\[
\mathcal{C} = \{B_i\,:\, i=1,2,\ldots\}Ê\cup \{B^n(z,\dist(z,\partial D)/3)\,:\, z\in D\}
\]
is a covering of $\overline{D}$ with open balls. By compactness of $\overline{D}$, there is a finite subfamily
$\mathcal{F}\subset \mathcal{C}$ such that
$\overline{D}\subset \cup_{B\in\mathcal{F}} B$. We define $u=\min\{1,g\}$, where
\[
g = \sum_{B\in\mathcal{F}}  \psi^B\,.
\]
Now $u\in C_0(G)$ and $0\le u \le 1$. Also, $u(x)=1$ for every $x\in\overline{D}$. Namely, 
if $x\in\overline{D}$, there exists $B_x\in\mathcal{F}$ such that $x\in B_x$ and so $g(x)\ge \psi^{B_x}(x)=1$.
Hence, $u(x)=1$. 

By definition of the function $u$ 
\begin{align*}
\int_{G\setminus D}\int_{G\setminus D} \frac{\lvert u(x)-u(y)\rvert}{\lvert x-y\rvert^{n+\delta}}\,dy\,dx
&\le \int_{G\setminus D}\int_{G\setminus D} \frac{\lvert g(x)-g(y)\rvert}{\lvert x-y\rvert^{n+\delta}}\,dy\,dx\\
&\le \sum_{B\in\mathcal{F}}\int_{G\setminus D}\int_{G\setminus D} \frac{\lvert \psi^B(x)-\psi^B(y)\rvert}{\lvert x-y\rvert^{n+\delta}}\,dy\,dx\,.
\end{align*}
We note that $\psi^B(x)-\psi^B(y)=0$ if $x,y\in G\setminus D$ and $B=B^n(z,\dist(z,\partial D)/3)$ for some
$z\in D$. Hence, by estimates in \eqref{e.small}
\begin{align*}
\int_{G\setminus D}\int_{G\setminus D} \frac{\lvert u(x)-u(y)\rvert}{\lvert x-y\rvert^{n+\delta}}\,dy\,dx&\le \sum_{i=1}^\infty\int_{\R^n}\int_{\R^n} \frac{\lvert \psi^{B_i}(x)-\psi^{B_i}(y)\rvert}{\lvert x-y\rvert^{n+\delta}}\,dy\,dx\\
&\le \lvert \psi\rvert_{W^{\delta,1}(\R^n)}\sum_{i=1}^\infty (2\diam(E_i)+2^{-i}\varepsilon)^{n-\delta}\\&\le C(\psi,n,\delta)(\varepsilon + \varepsilon^{n-\delta})\,.
\end{align*}
The lemma is proved.
\end{proof}

We are ready to complete the proof of Theorem \ref{l.regular}.

\begin{proof}[Proof of Theorem \ref{l.regular}]
Suppose that $D\subset\subset G$ is an open set and 
Ê$\mathcal{H}^{n-\delta}(\partial D)=0$. 
Let $\varepsilon>0$ and let $u=u_\varepsilon$
be the $C_0(G)$ function given by Lemma \ref{l.regular_est}.
Then, we obtain
\begin{align*}
\mathrm{cap}_{\delta,1}(\overline{D},G)&\le \int_G \int_G \frac{\lvert u(x)-u(y)\rvert}{\lvert x-y\rvert^{n+\delta}}\,dy\,dx \\
&\le 2\int_D \int_{G\setminus D}Ê\frac{1}{\lvert x-y\rvert^{n+\delta}}\,dy\,dx + \int_{G\setminus D}\int_{G\setminus D}\frac{\lvert u(x)-u(y)\rvert}{\lvert x-y\rvert^{n+\delta}}\,dy\,dx \\
&< Ê2P_\delta(D,G) + \varepsilon\,.
\end{align*}
The theorem is proved by taking $\varepsilon\to 0$. 
\end{proof}

\begin{cor}\label{c.regular}
Suppose that  $u\in C^\infty_0(G)$.
Let $0<\delta <1$ be given.
Then the set \[D_t:=\{x\in G\,:\, u(x)> t\}\] satisfies inequality \eqref{frac_cap_per} with respect to $G$  for almost every $t>0$.
\end{cor}

\begin{proof}
We note that $D_t\subset\subset G$ and
$\partial D_t \subset  \{x\in G\,:\,u(x)= t\}$
for every $t>0$.
Hence, by Lemma \ref{l.sard_type}, 
$\mathcal{H}^{n-1}(\partial D_t)<\infty$ for almost every $t>0$.
 By Remark \ref{r.comparion}, $ \mathcal{H}^{n-\delta}(\partial D_t)=0$ for almost every $t>0$.
 Thus,
 the claim follows from Theorem \ref{l.regular}.
\end{proof}

All quasiballs satisfy inequality \eqref{frac_cap_per} with respect to $\R^n$. 

\begin{example}\label{e.ex}
If
$f:\R^n\to \R^n$ is a $K$-quasiconformal mapping \cite[\S 3]{Astala}, then $D:=f(B^n(0,1))$ is called a quasiball.
We prove  that $D$ satisfies \eqref{frac_cap_per} with respect to $\R^n$ for every $0<\delta<1$. Let us 
write 
\[r^*=\inf\big\{ r>0\,:\, D\subset \partial D+B^n(0,r)\big\}<\infty\,.\] Then, by \cite[Theorem 1.3]{Faraco} there is a constant $C=C(n,\delta,r^*,K)>0$ such that
\begin{equation}\label{e.lasku}
\int_0^{r^*} \lvert \partial D + B^n(0,r)\rvert\, \frac{dr}{r^{1+\delta}} \le C \big\{\lvert D\rvert +\lvert \chi_D\rvert_{W^{\delta,1}(\R^n)}\big\}\,.
\end{equation}
In particular, if $\mathcal{M}^{n-\delta}_\ast(\partial D)>0$, then $2\,P_\delta(D,\R^n) = \lvert \chi_D\rvert_{W^{\delta,1}(\R^n)}=\infty$\,.
Hence inequality
\begin{equation}\label{e.quasi}
\mathrm{cap}_{\delta,1}(\overline{D},\R^n) \le 2\,P_\delta(D,\R^n)
\end{equation}
holds. On the other hand, if  $\mathcal{M}^{n-\delta}_\ast(\partial D)=0$, then
$\mathcal{H}^{n-\delta}(\partial D)=0$ and 
 inequality \eqref{e.quasi} holds by Theorem \ref{l.regular}. Thus,
$D$ satisfies inequality \eqref{frac_cap_per} with respect to $\R^n$.\qed
\end{example}


\section{The main result}\label{m_proof}

Theorem \ref{t.equiv} is a consequence of the following, more general, result.

\begin{thm}\label{t.equiv_II}
Suppose that  $G$ is an open set in $\R^n$. Let
$1\le q<\infty$, $\delta\in (0,1)$, and a constant $C>0$ be given. Then the following conditions are equivalent.
\begin{itemize}
\item[(A)]  The fractional inequality
\[
\bigg(\int_G \lvert u(x)\rvert^{q}\,dx\bigg)^{1/q} \le C 
\int_G\int_G\frac{\lvert u(x)-u(y)\rvert}{\lvert x-y\rvert^{n+\delta}}\,
dy\,dx
\]
holds for all measurable functions $u:G\to \R$ with compact support in $G$.
\item[(B)] The fractional isocapacitary inequality
\[
\lvert K\rvert^{1/q} \le C\,\mathrm{cap}_{\delta,1}(K,G)
\]
holds for every compact set $K$ in $G$.
\item[(C)] 
The fractional isoperimetric inequality
\[
\lvert D\rvert^{1/q} \le 2C\,P_\delta(D,G)
\]
holds for every  open set $D\subset\subset G$ whose boundary $\partial D$  is an $(n-1)$-dimensional
$C^\infty$-manifold in $G$. 
\end{itemize}
\end{thm}

We first give an immediate consequence of condition (A) in Theorem \ref{t.equiv_II}.

\begin{rem}
Let  $G$ be  an open set in $\R^n$. If $q\in [1, \infty )$ and $\delta\in (0,1)$ are given such that condition (A) in Theorem \ref{t.equiv_II} holds with a constant $C>0$, then
the inequality
\begin{equation}\label{remark_1}
\lvert A\rvert^{1/q} \le 2C\,P_\delta(A,G)
\end{equation}
holds for every measurable set $A\subset\subset G$.
This follows from condition (A) when $u=\chi_A$.
\end{rem}

For the proof of Theorem \ref{t.equiv_II} we need some auxiliary results.
First we recall an extension of the classical coarea formula
\begin{equation}\label{e.coarea_classical}
\int_{\R^n} \lvert \nabla u(x)\rvert\,dx = \int_{-\infty}^\infty \mathcal{H}^{n-1}(\{u = t\})\,dt\,,
\end{equation}
which is valid for every real-valued Lipschitz function $u$ on $\R^n$, we refer to \cite[\S 3.2]{Fed}.
The following fractional coarea formula is from \cite[Lemma 10]{APM}.

\begin{lem}\label{l.coarea}
Suppose that $G$ is an open set in $\R^n$. Let $0<\delta<1$ be given.
Then
\begin{equation}\label{e.coarea}
\frac{1}{2} \lvert u\rvert_{W^{\delta,1}(G)} = \int_0^\infty P_{\delta} (\{u> t\},G)\,dt
\end{equation}
for every $u:G\to [0,\infty)$ with $u\in \dot W^{\delta,1}(G)$.
\end{lem}

\begin{proof}
We note that
\[
\lvert u(x)-u(y)\rvert = \int_0^\infty  \lvert \chi_{\{u>t\}}(x)-\chi_{\{u>t\}}(y)\rvert\,dt
\]
for every $x,y\in G$.
On the other hand,
\[
\lvert \chi_{\{u>t\}}(x)-\chi_{\{u>t\}}(y)\rvert = \chi_{\{u>t\}}(x) \chi_{G\setminus \{u>t\}}(y)
+\chi_{\{u>t\}}(y) \chi_{G\setminus \{u>t\}}(x)\,.
\]
Hence, by Fubini's theorem
\begin{align*}
\lvert u\rvert_{W^{\delta,1}(G)} &= \int_G \int_G \int_0^\infty \frac{\lvert \chi_{\{u>t\}}(x)-\chi_{\{u>t\}}(y)\rvert}{\lvert x-y\rvert^{n+\delta}}\,dt\,dy\,dx\\
&=2 \int_0^\infty \int_{\{u>t\}}\int_{G\setminus \{u>t\}} \frac{1}{\lvert x-y\rvert^{n+\delta  }}\,dx\,dy\,dt
=2\int_0^\infty P_\delta(\{u>t\},G)\,dt\,.
\end{align*}
\end{proof}

We prove an approximation lemma. Let $\varphi\in C^\infty_0(B^n(0,1))$ be a non-negative bump
function such that
\[
\int_{\R^n}Ê\varphi(x)\,dx = 1\,.
\]
For $j\in\mathbf{N}$ and $x\in\R^n$, we write $\varphi_j(x)=2^{jn}\varphi(j x)$.
If $u\in L^p(\R^n)$ and $1\le p<\infty$,
it is well known that $u\ast \varphi_j\to u$ in $L^p(\R^n)$ when $j\to\infty$.
We use this fact in the proof of the following lemma
which tells that
the standard mollification converges to $u$ in the fractional seminorm
$\lvert \cdot \rvert_{W^{\delta,1}(G)}$.

\begin{lem}\label{l.approx}
Suppose that $G$ is an open set in $\R^n$.
Let $0<\delta<1$ be given.
Let
$u:G\to \R$ be a function in $\dot W^{\delta,1}(G)$ with
compact support in $G$. Then, 
\[
\lvert u-u\ast \varphi_j\rvert_{W^{\delta,1}(G)} \xrightarrow{j\to\infty}Ê0\,.
\]
\end{lem}

\begin{proof}
Let us fix $\varepsilon>0$ and let $K$
denote the support of $u$. Then $K$ is a compact set in $G$ and therefore
$d=\dist(K,\partial G)>0$. We write
\begin{align*}
\lvert u-u\ast \varphi_j\rvert_{W^{\delta,1}(G)}
= \int_G \int_G \frac{\lvert u(x)-u\ast \varphi_j(x) -u(y)+u\ast \varphi_j(y)\rvert}{\lvert x-y\rvert^{n+\delta}}\,dy\,dx\,.
\end{align*}
Since $\lvert u\rvert_{W^{\delta,1}(G)}<\infty$, we may  apply the monotone convergence theorem in $\R^n\times \R^n$ in order to obtain a number
$\rho<d$ such that
\[
\int_G \int_{G\cap B^n(x,\rho)} \frac{\lvert u(x)-u(y)\rvert}{\lvert x-y\rvert^{n+\delta}}\,dy\,dx <\varepsilon\,.
\]
Now, for any $j\in\mathbf{N}$
\begin{align*}
\int_G \int_{G\cap B^n(x,\rho)} &\frac{\lvert u\ast \varphi_j(x)-u\ast \varphi_j(y)\rvert}{\lvert x-y\rvert^{n+\delta}}\,dy\,dx \\&\le \int_{\R^n}Ê\varphi_j(z) \int_G \int_{G\cap B^n(x,\rho)} \frac{\lvert u(x-z)-u(y-z)\rvert}{\lvert x-y\rvert^{n+\delta}}\,dy\,dx\,dz \\
&= \int_{\R^n}Ê\varphi_j(z) \int_{G-z} \int_{(G-z)\cap B^n(x,\rho)} \frac{\lvert u(x)-u(y)\rvert}{\lvert x-y\rvert^{n+\delta}}\,dy\,dx\,dz\,.
\end{align*}
Since $\rho<d=\dist(K,\partial G)=\dist(K,\R^n\setminus G)$, we obtain that
\begin{align*}
\int_G \int_{G\cap B^n(x,\rho)} \frac{\lvert u\ast \varphi_j(x)-u\ast \varphi_j(y)\rvert}{\lvert x-y\rvert^{n+\delta}}\,dy\,dx
\le \int_G\int_{G\cap B^n(x,\rho)}\frac{\lvert u(x)-u(y)\rvert}{\lvert x-y\rvert^{n+\delta}}\,dy\,dx< \varepsilon\,.
\end{align*}
Hence, 
\[
\int_G \int_{G\cap B^n(x,\rho)} \frac{\lvert u(x)-u\ast \varphi_j(x) -u(y)+u\ast \varphi_j(y)\rvert}{\lvert x-y\rvert^{n+\delta}}\,dy\,dx < 2\varepsilon\,.
\]
On the other hand, by Lemma \ref{l.integrability} we have that $u\in L^1(\R^n)$ and therefore
\begin{align*}
&\int_G \int_{G\setminus B^n(x,\rho)} \frac{\lvert u(x)-u\ast \varphi_j(x) -u(y)+u\ast \varphi_j(y)\rvert}{\lvert x-y\rvert^{n+\delta}}\,dy\,dx \\
&\qquad \le 2\bigg(\int_{\R^n\setminus B^n(0,\rho)} \lvert x\rvert^{-n-\delta}\,dx \bigg) \cdot\int_G \lvert u(x)-u\ast \varphi_j(x)\rvert\,dx\xrightarrow{j\to\infty} 0\,.
\end{align*}
The claim follows by combining these estimates.
\end{proof}

We are ready to prove Theorem \ref{t.equiv_II}.

\begin{proof}[Proof of Theorem \ref{t.equiv_II}]
The implication from (A) to (B) is clear. 
Let us prove the implication from (B) to (C).
Let $D\subset\subset G$ be an open set whose boundary $\partial D$ is 
an $(n-1)$-dimensional $C^\infty$-manifold in $G$.
By condition (B), Lemma \ref{l.sard_type},
Remark \ref{r.comparion}, and Theorem \ref{l.regular}, we obtain that
\[
\lvert D\rvert^{1/q}\le \lvert \overline{D}\rvert^{1/q}\le C\, \mathrm{cap}_{\delta,1}(\overline{D},G)\le 2 C\,P_\delta(D,G)\,.
\]
This implies condition (C). 

Let us prove the implication
from (C) to (A). 
We fix a measurable function $u:G\to \R$ with compact support in $G$.
Without loss of generality, we may assume that $u\in \dot W^{\delta,1}(G)$.  
By $1$-Lipschitz truncation and monotone convergence theorem, we may assume that $u$ is bounded.
In particular,  $u\in L^{q}(G)$.
Let us write $u_j=\lvert u\rvert \ast \varphi_j\ge 0$.
%
Since $\lvert u\rvert \in L^{q}(G)$ and $\lvert u\rvert$ has a compact support in $G$,
\begin{align*}
\bigg(\int_G \lvert u(x)\rvert^{q}\,dx\bigg)^{1/q} = \lim_{j\to\infty}Ê\bigg(\int_G u_j(x)^{q}\,dx\bigg)^{1/q}\,.
\end{align*}
We follow an argument given in  \cite[\S 7]{L} and we focus on sufficiently large values of $j$ so that $u_j\in C^\infty_0(G)$.
By  Minkowski's integral inequality,
\begin{align*}
\bigg(\int_G u_j(x)^{q}\,dx\bigg)^{1/q}&=\bigg(\int_{G}Ê\bigg(\int_0^\infty \chi_{\{ u_jÊ> t\}}(x)\,dt\bigg)^{q}\,dx\bigg)^{1/q}
\\&\le \int_0^\infty \bigg(\int_{G} \chi_{\{  u_jÊ> t\}}(x) \,dx\bigg)^{1/q}\,dt
= \int_0^\infty \lvert \{  u_jÊ> t\} \rvert^{1/q}\,dt\,.
\end{align*}
By Sard's theorem \cite{Sard} for almost
every $t>0$,  the gradient  of 
$u_j$ differs from zero at every point in the level set $\{x\in G\,:\, u_j(x)=t\}$.
For these particular values of $t>0$,  the  boundary of an open set \[\{x\in G\,:\,  u_j(x)Ê> t\}\subset\subset G\]
coincides with the level set $\{x\in G\,:\, u_j(x)=t\}$ and, moreover, this level set is a compact $(n-1)$-dimensional $C^\infty$-manifold in $G$ by the implicit function theorem.

Hence, by condition (C), Lemma \ref{l.coarea},  and Lemma \ref{l.approx}
applied to $\lvert u\rvert\in \dot{W}^{\delta,1}(G)$, we obtain
\begin{align*}
&\bigg(\int_G  u_j(x)^{q}\,dx\bigg)^{1/q}\\&\le 2C\int_0^\infty P_\delta(\{  u_j  > t\},G)\,dt = \frac{2C}{2}\lvert u_j  \rvert_{W^{\delta,1}(G)}\xrightarrow{j\to\infty} C\big\lvert \lvert u\rvert \big\rvert_{W^{\delta,1}(G)}\le C\lvert u\rvert_{W^{\delta,1}(G)}\,.
\end{align*}
Condition (A) follows. 
\end{proof}

\end{document}